\documentclass[11pt]{article}

\usepackage{epsfig, graphicx}

\usepackage{amsthm,amsmath,amsmath,amssymb,amsthm,amsfonts,amstext,amsbsy,amscd}
\usepackage[titletoc,title]{appendix}

\usepackage[numbers]{natbib}

\RequirePackage[OT1]{fontenc}
\usepackage[applemac]{inputenc}

\usepackage{mathrsfs, amscd, amsfonts}

\usepackage{color}
\usepackage{dsfont} 
\usepackage{soul} 
\usepackage{graphicx} 
\usepackage{multirow}

\numberwithin{equation}{section}

\theoremstyle{plain}
\newtheorem{theorem}{Theorem}[section]
\newtheorem{definition}[theorem]{Definition}
\newtheorem{corollary}[theorem]{Corollary}
\newtheorem{lemma}[theorem]{Lemma}
\newtheorem{prop}[theorem]{Proposition}
\newtheorem{assumption}[theorem]{Assumption}{\bf}{\rm}
\theoremstyle{remark}
\newtheorem{remark}[theorem]{Remark}

\newcommand{\field}[1]{\mathbb{#1}}
\newcommand{\EE}{\field{E}}

\newcommand{\GG}{\field{G}}

\newcommand{\NN}{\field{N}}
\newcommand{\PP}{\field{P}}

\newcommand{\RR}{\field{R}}
\newcommand{\TT}{\field{T}}

\newcommand{\Bb}{{\mathcal B}}

\newcommand{\Ee}{{\mathcal E}}
\newcommand{\Ff}{{\mathcal F}}

\newcommand{\Ll}{{\mathcal L}}
\newcommand{\Mm}{{\mathcal M}}

\newcommand{\Pp}{{\mathcal P}}
\newcommand{\Qq}{{\mathcal Q}}

\def\<{\langle}
\def\>{\rangle}

\DeclareMathAlphabet{\mathpzc}{OT1}{pzc}{m}{it}

\makeatother
\begin{document}
\title{Transportation cost-information and concentration inequalities for bifurcating Markov chains}

\author{S. Val\`ere Bitseki Penda\thanks{CMAP, Ecole polytechnique
route de Saclay, 91128, Palaiseau, France. {\tt email}:
bitseki@cmap.polytechnique.fr} \and Mikael Escobar-Bach \thanks{Department of Mathematics and Computer Science, University of Southern Denmark, Campus 55,
5230 Odense M, Denmark {\tt email}:escobar@sdu.dk} \and Arnaud Guillin
\thanks{Institut Universitaire de France et Laboratoire de Math\'ematiques,
CNRS UMR 6620, Universit\'e Blaise Pascal, 24 avenue des Landais, BP
80026, 63177 Aubi\`ere, France. {\tt email}:
guillin@math.univ-bpclermont.fr}  }

\date{\today}
\maketitle

\begin{abstract}
We investigate the transportation cost-information inequalities for
bifurcating Markov chains which are a class of processes indexed by
binary tree. These processes provide models for cell growth when
each individual in one generation gives birth to two offsprings in
the next one. Transportation cost inequalities provide useful
concentration inequalities. We also study deviation inequalities for
the empirical means under relaxed assumptions on the Wasserstein
contraction of the Markov kernels. Applications to bifurcating non
linear autoregressive processes are considered: deviation
inequalities for pointwise estimates of the non linear leading
functions.\end{abstract}



\section{Introduction}

Roughly speaking, a bifurcating Markov chain is a Markov chain
indexed by a binary regular tree. This class of processes are well
adapted for the study of populations where each individual in one
generation gives birth to two offsprings in the next one. They were
introduced by Guyon \cite{Guyon} in order to study the
\textit{Escherichia coli} aging process.  Namely, when a cell divides into two offsprings, are the genetical traits identical for the two daughter cells? Recently, several models of
bifurcating Markov chains, or models using the theory of bifurcating
Markov chains, for example under the form of bifurcating
autoregressive processes, have been studied \cite{BH99,BH00,Guyon,
DHKR, DSGPM14}, showing that these processes are of great importance
to analysis of cell division. There is now an important literature
covering asymptotic theorems for bifurcating Markov chains such as
Law of Large Numbers, Central Limit Theorems, Moderate Deviation
Principle, Law of Iterated Logarithm, see for example \cite{Guyon,
Gu&Al,BDSGP,DM,SGM12,DSGPM14,BDG14} for recent references. These
limit theorems are particularly useful when applied to the
statistics of the bifurcating processes, enabling to provide
efficient tests to assert if the aging of the cell is different for
the two offsprings (see \cite{Gu&Al} for real case study). Of
course, these limit theorems may be considered only in the "ergodic"
case, i.e. when the law of the random lineage chain has an unique
invariant measure.

However, limit theorems are only asymptotical results and one is
often faced to study only datas with a size limited population.
It is thus very natural to control the statistics non
asymptotically. Such deviation inequalities (or concentration
inequalities) have been recently the subject of many studies and we
refer to the books of Ledoux \cite {Le01} and Massart \cite{Ma07}
for nice introductions on the subject, developing both i.i.d. case
and dependent case with a wide variety of tools (Laplace controls,
functional inequalities, Efron-Stein,...). It was one of the goal of
Bitseki et \textit{al.} \cite{BDG14} to investigate deviation
inequalities for additive functionals of bifurcating Markov chain.
In their work, one of the main hypothesis is that the Markov chain
associated to a random lineage of the population is uniformly
geometrically ergodic. It is clearly a very strong assumption,
nearly reducing interesting models to the compact case. The purpose
of this paper is to considerably weaken this hypothesis. More specifically, our
aim is to obtain deviation inequalities for bifurcating Markov chain
when the auxiliary Markov chains may satisfy some contraction
properties in Wasserstein distance, and some (uniform) integrabilty
property. This will be done with the help of transportation
cost-information inequalities and direct Laplace controls.
In order to present our result, we may now define properly the model of bifurcating Markov chains.

\subsection{Bifurcating Markov chains}

First we introduce some useful notations. Let $\TT$ be a regular
binary tree in which each vertex is seen as a positive integer
different from 0. For $r\in\mathbb{N}$, let
\begin{equation*}
\mathbb{G}_{r}=\Big\{2^{r},2^{r}+1,\cdots,2^{r+1}-1\Big\}, \quad
\mathbb{T}_{r}=\bigcup\limits_{q=0}^{r}\mathbb{G}_{q},
\end{equation*}
which denote respectively the $r$-th column and the first $(r+1)$
columns of the tree. The whole tree is thus defined by
\begin{equation*}
\TT = \bigcup_{r=0}^{\infty} \GG_{r}.
\end{equation*}
A column of a given vertex $n$ is $\GG_{r_{n}}$ with
$r_{n}=\lfloor\log_{2}n\rfloor$, where $\lfloor x\rfloor$ denotes
the integer part of the real number $x$.

In the sequel, we will see $\TT$ as a given population in which each
individual in one generation gives birth to two offsprings in the
next one. This will make easier the introduction of different
notions. The vertex $n$ will denote the individual $n$ and the
ancestor of individuals $2n$ and $2n+1$. The individuals who belong
to $2\NN$ (resp. $2\NN +1$) will be called individual of type 0
(resp. type 1). The column $\GG_{r}$ and the first $(r+1)$ columns
$\TT_{r}$ will denote respectively the $r$-th generation and the
first $(r+1)$ generations. The initial individual will be denoted
$1$.

For each individual $n$, we look into a random variable $X_{n}$,
defined on a probability space $(\Omega,\Ff,\PP)$ and which takes
its values in a metric space $(E,d)$ endowed with its Borel
$\sigma$-algebra $\Ee$. We assume that each pair of random variables
$(X_{2n},X_{2n+1})$ depends of the past values $(X_{m},
m\in\TT_{r_{n}})$ only through $X_{n}$. In order to describe this
dependance, let us introduce the following notion.

\begin{definition}[$\mathbb{T}$-transition probability, see (\cite{Guyon})]
We call $\mathbb{T}$-transition probability any mapping $P: E\times
\Ee^{2}\rightarrow [0,1]$ such that
\begin{itemize}
\item $P(\cdot,A)$ is measurable for all $A\in \Ee^{2}$,

\item $P(x,\cdot)$ is a probability measure on $(E^{2},\Ee^{2})$
for all $x\in E$.
\end{itemize}
\end{definition}

In particular, for all $x,y,z\in E$, $P(x,dy,dz)$ denotes the
probability that the couple of the quantities associated with the
children are in the neighbourhood of $y$ and $z$ given that the
quantity associated with their mother is $x$.

For a $\TT$-transition probability $P$ on $E\times \Ee^{2}$, we
denote by $P_{0}$, $P_{1}$ the first and the second marginal of $P$,
that is $P_{0}(x,A)=P(x,A\times E)$, $P_{1}(x,A)=P(x,E\times A)$ for
all $x\in E$ and $A\in \Ee$. Then, $P_{0}$ (resp. $P_{1}$) can be
seen as the transition probability associated to individual of type
0 (resp. type 1).

For $p\geq 1$, we denote by $\Bb(E^{p})$ (resp. $\Bb_{b}(E^{p})$),
the set of all $\Ee^{p}$-measurable (resp. $\Ee^{p}$-measurable and
bounded) mappings $f: E^{p}\rightarrow \mathbb{R}$. For $f\in
\Bb(E^{3})$, we denote by $Pf \in \Bb(E)$ the function
\begin{equation*}
x\mapsto Pf(x)=\int_{S^{2}}f(x,y,z)P(x,dy,dz), \,\,{\rm  when\, it\,
is\, defined}.
\end{equation*}

We are now in position to give a precise definition of bifurcating
Markov chain.

\begin{definition}[Bifurcating Markov Chains, see (\cite{Guyon})]
Let $(X_{n}, n\in \TT)$ be a family of $E$-valued random variables
defined on a filtered probability space $(\Omega, \Ff, (\Ff_{r},
r\in\NN), \PP)$. Let $\nu$ be a probability on $(E, \Ee)$ and $P$ be
a $\TT$-transition probability. We say that $(X_{n}, n\in \TT)$ is a
$(\Ff_{r})$-bifurcating Markov chain with initial distribution $\nu$
and $\TT$-transition probability $P$ if
\begin{itemize}
\item $X_{n}$ is $\Ff_{r_{n}}$-measurable for all $n\in
\TT$,
\item $\Ll(X_{1})=\nu$,
\item for all $r\in \NN$ and for all family $(f_{n}, n\in
\GG_{r})\subseteq \Bb_{b}(E^{3})$
\begin{equation*}
\EE\left[\prod\limits_{n\in
\GG_{r}}f_{n}(X_{n},X_{2n},X_{2n+1})\Big|
\Ff_{r}\right]=\prod\limits_{n\in \GG_{r}}Pf_{n}(X_{n}).
\end{equation*}
\end{itemize}
\end{definition}

In the following, when unprecised, the filtration implicitly used
will be $\Ff_r=\sigma(X_i, i\in {\mathbb T}_r)$.

\begin{remark}
We may of course also consider in this work bifurcating Markov chains on a $a$-ary tree (with $a\ge 2$) with no additional technicalities, but heavy additional notations. In the same spirit, Markov chains of higher order (such as BAR processes considered in \cite{BiDj12}) could be handled by the same techniques. A non trivial extension would be the case of bifurcating Markov chains on a Galton-Watson tree (see for example \cite{BP} under very strong assumptions), that we will consider elsewhere.
\end{remark}


\subsection{Transportation cost-information inequality}

We recall that $(E,d)$ is a metric space endowed with its Borel
$\sigma$-algebra $\Ee$. Given $p\geq 1$, the $L^{p}$-Wasserstein
distance between two probability measures $\mu$ and $\nu$ on $E$ is
defined by
\begin{equation*}\label{eq:Lpwassertein}
W_{p}^{d}(\nu,\mu) = \inf\left(\int\int
d(x,y)^{p}d\pi(x,y)\right)^{1/p},
\end{equation*}
where the infimum is taken over all probability measures $\pi$ on
the product space $E\times E$ with marginal distributions $\mu$ and
$\nu$ (say, coupling of $(\mu,\nu)$). This infimum is finite as soon
as $\mu$ and $\nu$ have finite moments of order $p$. When
$d(x,y)=\mathds{1}_{x\neq y}$ (the trivial measure),
$2W_{1}^{d}(\mu,\nu) = \|\mu-\nu\|_{TV}$, the total variation of
$\mu-\nu$.

The Kullback information (or relative entropy) of $\nu$ with respect
to $\mu$ is defined as
\begin{equation*}\label{eq:kullback}
H(\nu/\mu) = \begin{cases} \int \log \frac{d\nu}{d\mu} d\nu, \quad
\text{if $\nu\ll\mu$} \\ +\infty \hspace{1.55cm} \qquad \text{else}.
\end{cases}
\end{equation*}

\begin{definition}[$L^{p}$-transportation cost-inequality]
We say that the probability measure $\mu$ satisfies the
$L^{p}$-transportation cost-information inequality on $(E,d)$ (and
we write $\mu\in T_{p}(C)$) if there is some constant $C>0$ such
that for any probability measure $\nu$,
\begin{equation*}\label{eq:transport_inequality}
W_{p}^{d}(\mu,\nu) \leq \sqrt{2C H(\nu/\mu)}.
\end{equation*}
\end{definition}

This transportation inequality have been introduced by Marton
\cite{Mar86,Mar96} as a tool for (Gaussian) concentration of measure
property. The following result will be crucial in the sequel. It
gives a characterization of $L^{1}$-transportation cost-inequality
in term of concentration inequality. It is of course one of the main
tool to get deviation inequalities (via Markov inequality).

\begin{theorem}[\cite{BG99}]\label{thm:bobkov-gotze}
$\mu$ satisfies the $L^{1}$-transportation cost-information
inequality (say $T_1$) on $(E,d)$ with constant $C>0$, that is, $\mu\in
T_{1}(C)$, if and only if for any Lipschitzian function $F: (E,d)
\rightarrow \RR$, $F$ is $\mu$-integrable and
\begin{equation*}\label{eq:bobkov_gotze}
\int_{E} \exp\left(\lambda\left(F-\langle
F\rangle_{\mu}\right)\right) d\mu \leq
\exp\left(\frac{\lambda^{2}}{2}C\|F\|_{Lip}^{2}\right) \quad \forall
\lambda\in \RR,
\end{equation*}
where $\langle F\rangle_{\mu} = \int_{E}F d\mu$ and
\begin{equation*}\label{eq:norm_lip}
\|F\|_{Lip} = \sup_{x\neq y} \frac{|F(x)-F(y)|}{d(x,y)} < +\infty.
\end{equation*}
In particular, we have the concentration inequality
\begin{equation*}\label{eq:bobgot_con}
\mu\left(F-\langle F\rangle_{\mu}\leq -t\right) \vee
\mu\left(F-\langle F\rangle_{\mu}\geq t\right) \leq
\exp\left(-\frac{t^{2}}{2C\|F\|_{Lip}^{2}}\right) \quad \forall
t\in\RR.
\end{equation*}
\end{theorem}

In this work we will focus on transportation inequality $T_1$
mainly. There is now a considerable literature around these
transportation inequalities. As a flavor, let us cite first the
characterization of $T_1$ as a Gaussian integrability property
\cite{DGW04} (see also \cite{Go06}).
\begin{theorem}[\cite{DGW04}]\label{thm:dgw}
 $\mu$ satisfies the $L^{1}$-transportation cost-information
inequality (say $T_1$) on $(E,d)$ if and only if there exists
$\delta>0$ and $x_0\in E$ such that
$$\mu\left(e^{\delta d^2(x,x_0)}\right)<\infty,$$
and the constant of the Transportation inequality can be made
explicit.
\end{theorem}
There is also a large deviations characterization \cite{GL07}.
Recent striking results on transportation inequalities have been
obtained for $T_2$, namely that they are equivalent to dimension
free Gaussian concentration \cite{Go09}, or to a restricted class of
logarithmic Sobolev inequalities \cite{GRS11}. Se also \cite{CG06}
or \cite{CGW10} for practical criterion based on Lyapunov type
criterion and we refer for example to \cite{GL10} or \cite{Villani} for
surveys on transportation inequality. One of the main aspect  of
transportation inequality is their tensorization property, i.e.
$\mu^{\otimes n}$ will satisfy some transportation measure if $\mu$
does (with dependence on the dimension $n^{2/p-1}$) . One important development was to consider such a property for
dependent sequences such as Markov chains. In \cite{DGW04}, Djellout
et \textit{al.}, generalizing result of Marton \cite{Mar97}, have
provided conditions under which the law of a homogeneous Markov
chain $(Y_{k})_{1\leq k\leq n}$ on $E^{n}$ satisfies the
$L^{p}$-transportation cost-information inequality $T_p$ with
respect to the metric
\begin{equation*}
d_{l_{p}}(x,y) := \left(\sum_{i=1}^{n}
d(x_{i},y_{i})^{p}\right)^{1/p}.
\end{equation*}
We will follow similar ideas here to establish the
$L^{p}$- transportation cost-information inequality for the law of a
bifurcating Markov chain $(X_{i})_{1\leq i\leq N}$ on $E^{N}$. This
will allow us to obtain concentration inequalities for bifurcating
Markov chains under hypotheses largely weaker than those of Bitseki et
\textit{al.} \cite{BDG14}. It would also be tempting to generalize the approach of \cite{GLWY09} to Markov chains and bifurcating Markov chains to get directly deviation inequalities for Markov chains, w.r.t. the invariant measure. However it would need to restrict to reversible Markov chains and thus not directly suited to bifurcating Markov chains and would thus recquire new ideas.

\begin{remark}
There are natural generalizations of the $T_1$ inequality often
denoted $\alpha-T_1$ inequality, where $\alpha$ is a non negative
convex lower semi continuous function vanishing at 0. We say that
the probability measure $\mu$ satisfies $\alpha-T_p(C)$ if for any
probability measure $\nu$
$$\alpha\left(W_1(\nu,\mu)\right) \le 2C\, H(\nu/\mu).$$
The usual $T_1$ inequality is then the case where $\alpha(t)=t^2$.
Gozlan \cite{Go06} has generalized Bobkov-G\"otze's Laplace
transform control \cite{BG99} and Djellout-Guillin-Wu \cite{DGW04}
integrability criterion to this setting enabling to recover sub or
super Gaussian concentration. The result of the following section
can be generalized to this setting, however adding technical details
and heavy notations. Details will thus be left to the reader.
\end{remark}


\section{Transportation cost-information inequalities for
bifurcating Markov chains}\label{sec:trans_bmc}

Let $(X_{i},i\in\TT)$ be a bifurcating Markov chain on $E$ with
$\TT$-probability transition $P$ and initial measure $\nu$. For
$p\geq 1$ and $C>0$, we consider the following assumption that we
shall call ($H_{p}(C)$) in the sequel.
\begin{assumption}[$H_{p}(C)$]\label{ass:hp} $\,$
\begin{itemize}
\item [(a)] $\nu\in T_{p}(C)$;
\item [(b)] $P(x,\cdot,\cdot) \in T_{p}(C)$, $\forall x\in E$ ;
\item [(c)] $W_{p}^{d}(P(x,\cdot,\cdot),P(\tilde{x},\cdot,\cdot)) \leq
q \,d(x,\tilde{x})$, $\forall x, \tilde{x}\in E$ and some
$q>0$.
\end{itemize}
\end{assumption}
It is important to remark that under $(H_p(C)), (c)$ we have that
there exists $r_0$ and $r_1$ smaller than $q$ such that for $b=0,1$
$$W_{p}^{d}(P_{b}(x,\cdot),P_{b}(\tilde{x},\cdot)) \leq
r_b \,d(x,\tilde{x}),\qquad\qquad\forall x, \tilde{x}\in E.$$ Note
also that when $P(x,dy,dz)=P_0(x,dy)P_1(x,dz)$, then these last two
stability results in Wasserstein contraction implies $(H_p(C)), (c)$
with $q\le (r_0^p+r_1^p)^{1/p}$ (using trivial coupling). We may
remark also that by $(H_p(C)), (b)$, $P_0$ and $P_1$ also satisfies
(uniformly) a transportation inequality. Let us note that thanks to
the H\"{o}lder inequality, $(H_{p}(C))$
implies $(H_{1}(C))$. \\
We do not suppose here that $q$, $r_0$ and $r_1$ are strictly less than
1, and thus the two marginal chains, as well as the bifurcating one, are not  a priori contractions. We are thus
considering here both "stable" and "unstable" cases.

We then have the following result for the law of the whole
trajectory on the binary tree.

\begin{theorem}\label{prop:trans_bmc}
Let $n\in \NN$ and let $\Pp$ be the law of $(X_{i})_{1\leq i\leq
\TT_n}$ and denote $N=|\TT_n|$. We assume Assumption \ref{ass:hp}
for $1\leq p\leq2$. Then $\Pp \in T_{p}\left(C_{N}\right)$ where
\begin{equation*}\displaystyle
C_{N} = \begin{cases} C\frac{N^{2/p-1}}{\left(1 - q\right)^{2}}
\hspace{3.8cm} \text{if $\quad q<1$} \\  C
\exp\left(2-\frac{2}{p}\right) N^{2/p+1} \hspace{1.5cm} \text{if
$\quad q=1$} \\  C \left(N+1\right)
\left(\frac{\exp\left({q-1}\right)r^{pN}}{r^{p}-1}\right)^{2/p}
\hspace{0.51cm} \text{if $\quad q>1$}.
\end{cases}
\end{equation*}
\end{theorem}

Before the proof of this result, let us make the following
notations. For a Polish space $\chi$, we denote by $\Mm_{1}(\chi)$
the space of probability measures on $\chi$. For $x\in E^{N}$,
$x^{i} := (x_{1},\cdots,x_{i})$. For $\mu\in\Mm_{1}(E^{N})$, let
$(x_{1},\cdots,x_{N})\in E^{N}$ be distributed randomly according to
$\mu$. We denote by $\mu^{i}$ the law of $x^{2i+1}$, and by
$\mu_{x^{2i-1}}^{i}$ the conditional law of $(x_{2i},x_{2i+1})$
given $x^{2i-1}$ with the convention $\mu^{1}_{x^{0}} = \mu^{1}$,
where $x^{0}=x_{0}$ is some fixed point. In particular, if $\mu$ is
the law of a bifurcating Markov chain with $\TT$-probability
transition $P$, then $\mu_{x^{2i-1}}^{i} = P(x_{i},\cdot,\cdot)$.

For the convenience of the readers, we recall the formula of
additivity of entropy (see for e.g. \cite{Villani}, Lemma 22.8).
\begin{lemma}\label{lem:add_entropy}
Let $N\in \NN$, let $\chi_{1},\cdots,\chi_{N}$ be Polish spaces and
$\Pp, \Qq \in\Mm_{1}(\chi)$ where $\chi=\prod_{i=1}^{N}\chi_{i}$.
Then
\begin{equation*}
H(\Qq|\Pp) = \sum_{i=1}^{N} \int_{\chi}
H(\Qq_{x^{i-1}}^{i}|\Pp_{x^{i-1}}^{i})\Qq(dx)
\end{equation*}
where $\Pp_{x^{i-1}}^{i}$ and $\Qq_{x^{i-1}}^{i}$ are defined in the
same way as above.
\end{lemma}
We can now prove the Theorem.
\begin{proof}[Proof of the Theorem \ref{prop:trans_bmc}]
Let $\Qq\in\Mm_{1}(E^{N})$. Assume that $H(\Qq|\Pp)<\infty$ (trivial
otherwise). Let $\varepsilon>0$. The idea is of course to do a
conditionnement with respect to the previous generation, i.e. to $\GG_{n-1}$ but we will do it sequentially by pairs. Conditionally to their ancestors, every pair of offspring of an individual is independent of the offspring of the other individuals for the same generation. Let $i$ be a member of
generation $\GG_{j-1}$, and define for a realization $x$ on the tree
$\TT_i(x):=(x_1,...,x_{|\TT_j|})$. By the definition of the
Wasserstein distance, there is a coupling
$\pi_{y^{2i-1},x^{2i-1}}^{i}$ of
$(\Qq_{y^{2i-1}}^{i},\Pp_{x^{2i-1}}^{i})$ such that
\begin{eqnarray*}
\mathcal{A}_i&:=&\int (d(y_{2i},x_{2i})^p+d(y_{2i+1},x_{2i+1})^p)d\pi_{y^{2i-1},x^{2i-1}}^{i}\\
&\leq& (1+\epsilon)W^d_p\left(\Qq_{y^{2i-1}}^{i},\Pp_{x^{2i-1}}^{i}\right)^{p}\\
&\leq& (1 + \epsilon)\left[
W^d_p\left(\Qq_{y^{2i-1}}^{i},\Pp_{y^{2i-1}}^{i}\right) +
W^d_p\left(\Pp_{y^{2i-1}}^{i},\Pp_{x^{2i-1}}^{i}\right)\right]^p \\
&\leq& (1 + \epsilon)\left[ W^d_p\left(\Qq_{y^{2i-1}}^{i},
P\left(y_{i},\cdot,\cdot\right)\right) +
W^d_p\Big(P\left(y_{i},\cdot,\cdot\right),P\left(x_{i},\cdot,\cdot\right)\Big)\right]^p,
\end{eqnarray*}
where the second inequality is obtained thanks to the triangle
inequality for the $W_{p}^{d}$ distance and the equality is a
consequence of the Markov property.
By Assumption \ref{ass:hp}, and the convexity of the function
$x\mapsto x^{p}$, we obtain, for $a,b>1$ such that $1/a + 1/b=1$,
\begin{eqnarray*}
\mathcal{A}_i
&\leq& (1+\epsilon)\left(\sqrt{2CH_i(y^{2i-1})} + qd(y_{i},x_{i})\right)^p \\
&\leq& (1+\epsilon)\left(a^{p-1}\left(\sqrt{2CH_i(y^{2i-1})}\right)^p +
b^{p-1}q^{p} d^{p}(y_{i},x_{i})\right)
\end{eqnarray*}
where $H_i(y^{2i-1})=H(\Qq^i_{y^{2i-1}}|\mathcal{P}^i_{y^{2i-1}})$.
By recurrence, it leads to the finiteness of $p$-moments. Taking the
average with respect to the whole law and summing on $i$, we obtain
\begin{multline*}
\sum_{i=0}^{|\TT_{n-1}|} \EE(\mathcal{A}_i) \\ \leq \left (1 +
\varepsilon\right) \left(a^{p-1} \left(2C\right)^{p/2}
\sum_{i=1}^{|\TT_{n-1}|}
\EE\left[H_{i}(Y^{2i-1})^{p/2}\right]\right) + \left(b^{p-1}q^{p}
\sum_{i=0}^{|\TT_{n-2}|} \EE(\mathcal{A}_i)\right).
\end{multline*}
Letting $\varepsilon$ goes to $0^{+}$, we are led to
\begin{multline*}
\sum_{i=0}^{|\TT_{n-1}|} \EE(\mathcal{A}_i)  \\ \leq
 \sum_{i=1}^{N} \left(a^{p-1} \left(2C\right)^{p/2}
\EE\left[H_{i}(Y^{i-1})^{p/2}\right]\right) + \left(b^{p-1}q^{p}
\sum_{i=0}^{|\TT_{n-2}|} \EE(\mathcal{A}_i)\right).
\end{multline*}
Iterating the latter inequality, increasing some terms and thanks to
H\"{o}lder inequality, we obtain
\begin{multline*}
\sum_{i=0}^{|\TT_{n-1}|} \EE(\mathcal{A}_i) \leq \sum_{i=1}^{N}
\left(\sum_{j=1}^{i} h_{j}\right)\left(b^{p-1}q^{p}\right)^{N-i} =
\sum_{i=1}^{N} h_{i} \sum_{j=0}^{N-i} \left(b^{p-1}q^{p}\right)^{j}
\\ \leq \left(\sum_{i=1}^{N} h_{i}^{2/p}\right)^{p/2}
\left(\sum_{i=1}^{N} \left(\sum_{j=0}^{N-i}
\left(b^{p-1}q^{p}\right)^{j}\right)^{\frac{2}{2-p}}\right)^{\frac{2-p}{2}}
\end{multline*}
where $h_{i} = a^{p-1}(2C)^{p/2}\EE[H_{i}(Y^{i-1})^{p/2}]$. By the
definition of the Wasserstein distance, the additivity of entropy
and using the concavity of the function $x\mapsto x^{p/2}$ for
$p\in[1,2]$, we obtain
\begin{multline*}
W^{d_{l_{p}}}_p(\Qq,\mathcal{P})^{p} \leq a^{p-1}
\left(2CH\left(\Qq|\Pp\right)\right)^{p/2} \left(\sum_{i=1}^{N}
\left(\sum_{j=0}^{N-i}
\left(b^{p-1}q^{p}\right)^{j}\right)^{\frac{2}{2-p}}\right)^{\frac{2-p}{2}}
\\ \leq a^{p-1} \left(2CH\left(\Qq|\Pp\right)\right)^{p/2}
N^{1-\frac{p}{2}} \sum_{j=0}^{N-1} \left(b^{p-1}q^{p}\right)^{j}.
\end{multline*}
When $q<1$, we take $b=q^{-1}$, so that $b^{p-1}q^{p} = r<1$ and the
desired result follows easily. When $q\geq1$, we take $b=1+1/N$ and
the results follow from simple analysis and this ends the proof.
\end{proof}

\begin{remark}
For $q<1$, we then have that the constant $C_{N}$ of $T_{1}$ inequality for
$\Pp$ increases linearly on the dimension $N$. However, for $T_{2}$
this constant is independent of the dimension as in the i.i.d. case.
\end{remark}

\begin{remark}
As we will see in the next section, still when $q<1$, Theorem \ref{prop:trans_bmc} and
Theorem \ref{thm:bobkov-gotze} applied to $F(X_{1},\cdots,X_{N}) =
(1/N)\sum_{i=1}^{N} f(X_{i})$ (where $f$ is a Lipschitzian function
defined on $E$) gives us deviation inequalities with a good order of
$N$. But, when they are applied to $F(X_{1},\cdots,X_{N}) =
f(X_{N})$, deviation inequalities that we obtain does not furnish
the good order of $N$ when $N$ is large. The same remark holds when
$F(X_{1},\cdots,X_{N}) = g(X_{n},X_{2n},X_{2n+1})$ with
$n\in\{1,\cdots,(N-N[2])\}$ and $g$ a Lipschitzian function defined
on $E^{3}$. As this last question is important for the
$L^{1}$-transportation cost-information inequality of the invariant
measure of a bifurcating Markov chain, we give the following
results.
\end{remark}

\begin{prop}\label{prop:transport_Xs}
Under $(H_{1}(C))$, for any $n\in \TT$ and $x\in E$
\begin{equation*}
\Ll(X_{n}|X_{1}=x)\in T_{1}(c_{n})
\end{equation*}
where
\begin{equation*}
c_{n} = C \sum_{k=0}^{r_{n}-1} r_{0}^{2(k-a_{k})} r_{1}^{2a_{k}};
\quad a_{0} = 0
\end{equation*}
and for all $k\in\{1,\cdots,r_{n}-1\}$, $a_{k}$ is the number of
ancestor of type 1 of $X_n$ which are between the $r_{n}-k+1$-th
generation and the $r_{n}$-th generation.
\end{prop}

Before the proof, we introduce some more notations. Let $n\in\TT$. We denote
by $(z_{1},\cdots,z_{r_{n}})\in\{0,1\}^{r_{n}}$ the unique path from
the root 1 to $n$. Then, for all $i\in\{1,\cdots,r_{n}\}$, $z_{i}$
is the type of the ancestor of $n$ which is in the $i$-th generation
and the quantities $a_{k}$ defined in the Proposition
\ref{prop:transport_Xs} are given by
\begin{equation*}
a_{k} = \sum_{i=r_{n}-k+1}^{r_{n}} z_{i}.
\end{equation*}
For all $k\in\{1,\cdots,r_{n}\}$, we denote by $P^{k}$ and $P^{-k}$
the iterated of the transition probabilities $P_{0}$ and $P_{1}$
defined by
\begin{equation*}
P^{k} := P_{z_{1}}\circ \cdots \circ P_{z_{k}} \quad \text{and}
\quad P^{-k} := P_{z_{r_{n}-k}}\circ \cdots \circ P_{z_{r_{n}}}.
\end{equation*}

\begin{proof}[Proof of the Proposition \ref{prop:transport_Xs}]
First note that since
\begin{equation*}
W_{1}^{d}(\nu,\mu)= \sup_{f:\|f\|_{Lip}\leq
1}\left|\int_Sfd\mu-\int_Sfd\nu\right|,
\end{equation*}
condition \textit{(c)} of $(H_{1}(C))$ implies that
\begin{equation*}
\|P_bf\|_{Lip} \leq r_{b}\|f\|_{Lip} \quad \forall b\in\{0,1\}.
\end{equation*}
Now let $f$ be a Lipschitzian function defined on $E$. By
\textit{(b)}-\textit{(c)} of $(H_{1}(C))$ and Theorem
\ref{thm:bobkov-gotze}, we have
\begin{equation*}
P^{r_{n}}(e^{f}) \leq P^{r_{n}-1}\left(\exp\left( P_{r_{n}}f +
\frac{C\|f\|_{Lip}^2}{2}\right)\right).
\end{equation*}
Once again, applying Theorem \ref{thm:bobkov-gotze}, we obtain
\begin{equation*}
P^{r_{n}}(e^{f}) \leq P^{r_{n}-2}\left(\exp\left(
P^{-1}f+\dfrac{C\|f\|_{Lip}^2}{2}+\frac{C\|P_{z_{r_{n}}}f\|_{Lip}^2}{2}\right)\right).
\end{equation*}
By iterating this method, we are led to
\begin{equation*}
P^{r_{n}}(e^{f}) \leq \exp\left( P^{- r_{n} + 1}f + (1 +
r_{z_{r_{n}}}^2 + r_{z_{r_{n}}}^2r_{z_{r_{n} - 1}}^2+\cdots+
\prod_{i=2}^{r_{n}} r_{z_{i}}^2) \frac{C\|f\|_{Lip}^2}{2}\right).
\end{equation*}
Since
\begin{equation*}
1 + r_{z_{r_{n}}}^2 + r_{z_{r_{n}}}^2r_{z_{r_{n} - 1}}^2+\cdots+
\prod_{i=2}^{r_{n}} r_{z_{i}}^2 = \sum_{k=0}^{r_{n}-1}
r_{0}^{2(k-a_{k})} r_{1}^{2a_{k}} \quad \text{and} \quad P^{- r_{n}
+ 1}f = P^{r_{n}}f,
\end{equation*}
we conclude the proof thanks to Theorem \ref{thm:bobkov-gotze}.
\end{proof}

The next result is a consequence of the previous Proposition.

\begin{corollary}\label{cor:trans_X}
Assume ($H_{1}(C)$) and $r := \max\{r_{0}, r_{1}\} < 1$.
Then
\begin{equation*}
\Ll(X_{n}|{X_{1}=x})\in  T_{1}(c_{\infty}) \quad \text{and} \quad
\Ll((X_{n},X_{2n},X_{2n+1})|{X_{1}=x})\in T_1(c'_\infty)
\end{equation*}
where
\begin{equation*}
c_\infty = \frac{C}{1-r^{2}} \quad \text{and} \quad c'_\infty =
C\left(1 + \frac{(1+q)^{2}}{1-r^{2}}\right).
\end{equation*}
\end{corollary}

\begin{proof}
That $\Ll(X_{n}|{X_{1}=x})\in  T_{1}(c_{\infty})$ is a direct
consequence of Proposition \ref{prop:transport_Xs}. It suffices to
bound $r_{0}$ and $r_{1}$ by $r$.

In order to deal with the ancestor-offspring case
$(X_{n},X_{2n},X_{2n+1})$, we do the following remarks.

Let $f : (E^3,d_{l_{1}})\rightarrow \RR$ be a Lipschitzian function.
We have
\begin{equation*}
\|Pf\|_{Lip} = \sup_{x,\tilde{x}\in E} \frac{\left|\int
f(x,y,z)P(x,dy,dz) - \int f(\tilde{x},y,z)
P(\tilde{x},dy,dz)\right|}{d(x,\tilde{x})}.
\end{equation*}
Thanks to condition \textit{(c)} of $(H_{1}(C))$, we have the
following inequalities
\begin{align*}
\left|\int f(x,y,z)P(x,dy,dz) - \int
f(\tilde{x},y,z)P(\tilde{x},dy,dz)\right| \hspace{1.5cm} \\ \leq
\|f\|_{Lip}\left(d(x,\tilde{x})+W_1^{d_{l_1}}\left(P(x,\cdot),P(\tilde{x},\cdot)\right)\right)
\\  \leq (q+1)\|f\|_{Lip}d(x,\tilde{x}), \hspace{3.4cm}
\end{align*}
and then,
\begin{equation*}
\|Pf\|_{Lip} \leq (q+1)\|f\|_{Lip}.
\end{equation*}
We recall that $X_{1}=x$. We have
\begin{equation*}
\EE\left[\exp\left(f(X_{n},X_{2n},X_{2n+1})\right)\right] =
P^{r_{n}}(P e^{f} (x)).
\end{equation*}
Now, from $(H_{1}(C))$, the previous remarks and using the same
strategy as in the proof of Proposition \ref{prop:transport_Xs},  we
are led to
\begin{multline*}
\EE\left[\exp\left(f(X_{n},X_{2n},X_{2n+1})\right)\right] \\ \leq
\exp\left(P_{z_{1}}\cdots P_{z_{r_{n}}}P f(x) +
\frac{C\|f\|_{Lip}^{2}}{2} + \frac{C(1+q)^{2}\|f\|_{Lip}^{2}}{2}
\sum_{i=0}^{r_{n}-1} r^{2i}\right).
\end{multline*}
Since $P_{z_{1}}\cdots P_{z_{r_{n}}}Pf(x) = \EE\left[f(X_{n},
X_{2n}, X_{2n+1})\right]$ and $\sum_{i=0}^{r_{n}-1} r^{2i} \leq
1/(1-r^{2})$, we obtain
\begin{equation*}
\EE\left[\exp\left(f(X_{n},X_{2n},X_{2n+1})\right)\right] \leq
\exp\left(\EE\left[f(X_{n}, X_{2n}, X_{2n+1})\right] +
c'_{\infty}\right)
\end{equation*}
with $c'_{\infty}$ given in the Corollary. We then conclude the
proof thanks to Theorem \ref{thm:bobkov-gotze}.
\end{proof}


\section{Concentration inequalities for bifurcating Markov
chains}\label{sec:con_bmc}

\subsection{Direct applications of the Theorem \ref{prop:trans_bmc}}

We are now interested in the concentration inequalities for the
additive functionals of bifurcating Markov chains. Specifically, let
$N\in\NN^{*}$ and $I$ be a subset of $\{1,\cdots,N\}$. Let $f$ be a
real function on $E$ or $E^{3}$. We set
\begin{equation*}
M_{I}(f) = \sum_{i\in I} f(\Delta_{i})
\end{equation*}
where $\Delta_{i} = X_{i}$ if $f$ is defined on $E$ and $\Delta_{i}
= (X_{i},X_{2i},X_{2i +1})$ if $f$ is defined on $E^{3}$. We also
consider the empirical mean $\overline{M}_{I}(f)$ over $I$ defined
by $\overline{M}_{I}(f) = (1/|I|)M_{I}(f)$ where $|I|$ denotes the
cardinality of $I$. In the statistical applications, the cases
$N=|\TT_{n}|$ and $I=\GG_{m}$ (for $m\in\{0,\cdots,n\}$) or
$I=\TT_{n}$ are relevant (see for e.g. \cite{BDG14}).

First, we will establish concentration inequalities when $f$ is a
real Lipschitzian function defined on $E$. For a subset $I$ of
$\{1,\cdots,N\}$, let $F_{I}$ be the function defined on
$(E^{N},d_{l_{p}})$, $p\geq1$ by $F_{I}(x^{N}) = 1/(|I|) \sum_{i\in
I} f(x_{i})$ for all $x^{N}\in E^{N}$. Then $F_{I}$ is also a
Lipschitzian function on $(E^{N},d_{l_{p}})$ and we have
$\|F_{I}\|_{Lip} \leq |I|^{-1/p}\|f\|_{Lip}$. The following result
is a direct consequence of Theorem \ref{prop:trans_bmc}.

\begin{prop}\label{prop:concentration_Xi}
Let $N\in\NN^{*}$ and let $\Pp$ be the law of $(X_{i})_{1\leq i\leq
N}$. Let $f$ be a real Lipschitzian function on $(E,d)$. Then, under
$(H_{p}(C))$ for $1\leq p\leq2$,
\begin{equation*}
\Pp\circ F_{I}^{-1} \in T_{p}(C_{N}|I|^{-2/p}\|f\|_{Lip}^{2})
\end{equation*}
where $C_{N}$ is given in the Theorem \ref{prop:trans_bmc} and
$\Pp\circ F_{I}^{-1}$ is the image law of $\Pp$ under $F_{I}$. In
particular, for all $t>0$ we have
\begin{multline*}
\PP\left(F_{I}(X^{N})\leq -t + \EE\left[F_{I}(X^{N})\right]\right)
\vee \PP\left(F_{I}(X^{N})\geq t + \EE\left[F_{I}(X^{N})\right]\right) \\
\leq \exp\left(-\frac{t^{2}|I|^{2/p}}{2C_{N}\|f\|_{Lip}^{2}}\right).
\end{multline*}
\end{prop}

\begin{proof}
The first part is a direct consequence of Theorem
\ref{prop:trans_bmc} and Lemma 2.1 of \cite{DGW04}. The second part
is an application of Theorem \ref{thm:bobkov-gotze}.
\end{proof}
For the next concentration inequality, we assume that $f$ is a real
Lipschitzian function defined on $(E^{3},d_{l_{1}})$, which means
that
\begin{equation*}
|f(x)-f(y)|\leq \|f\|_{Lip}\sum_{i=1}^{3} d(x_{i},y_{i}) \quad
\forall x, y \in E^{3}.
\end{equation*}
We assume that $N$ is a odd number. Let $I$ be a subset of
$\{1,\cdots,(N-1)/2\}$. Now, we denote by $F_{I}$ the real function
defined on $(E^{N},d_{l_{p}})$ by $F_{I}(x^{N}) = (1/|I|) \sum_{i\in
I} f(x_{i}, x_{2i}, x_{2i+1})$. 
For all $x^{N}, y^{N} \in E^{N}$ we have for some universal constant $c$
\begin{align*}
|F_{I}(x^{N}) - F_{I}(y^{N})| &\leq \frac{\|f\|_{Lip}}{|I|}
\sum_{i\in I} \left(d(x_{i},y_{i}) + d(x_{2i},y_{2i}) +
d(x_{2i+1},y_{2i+1})\right)& \\ &\leq
\frac{c\|f\|_{Lip}}{|I|^{1/p}} d_{l_{p}}(x^{N},y^{N}).&
\end{align*}
$F_{I}$ is then a Lipschitzian function on $(E^{N},d_{l_{p}})$ and
$\|F_{I}\|_{Lip} \leq c\|f\|_{Lip}/|I|^{1/p}$. We
then have the following result.
\begin{prop}\label{prop:concentration_deltai}
Let $N\in\NN^{*}$ be a odd number and let $\Pp$ be the law of
$(X_{i})_{1\leq i\leq N}$. Let $f$ be a real Lipschitzian function
on $(E^{3},d_{l_{1}})$. Then, under $(H_{p}(C))$ for $1\leq p\leq2$,
\begin{equation*}
\Pp\circ F_{I}^{-1} \in
T_{p}(c\,C_{N}|I|^{-2/p}\|f\|_{Lip}^{2})
\end{equation*}
where $C_{N}$ is given in the Theorem \ref{prop:trans_bmc} and
$\Pp\circ F_{I}^{-1}$ is the image law of $\Pp$ under $F_{I}$. In
particular, for all $t>0$ we have
\begin{multline*}
\PP\left(F_{I}(X^{N})\leq -t + \EE\left[F_{I}(X^{N})\right]\right)
\vee \PP\left(F_{I}(X^{N})\geq t + \EE\left[F_{I}(X^{N})\right]\right) \\
\leq \exp\left(-\frac{t^{2}|I|^{2/p}}{2c  C_{N}
\|f\|_{Lip}^{2}}\right).
\end{multline*}
\end{prop}

\begin{proof}
The proof is a direct consequence of Theorem \ref{prop:trans_bmc},
Lemma 2.1 of \cite{DGW04} and Theorem \ref{thm:bobkov-gotze}.
\end{proof}

\begin{remark}
The previous results applyed with $p=1$ to the empirical means
$\overline{M}_{\GG_{n}}(f)$ and $\overline{M}_{\TT_{n}}(f)$ ($f$ being a real Lipschitzian function) give us relevant concentration
inequalities, that is with the good order size
of the index set, when $q < 1$. For example, for
$\overline{M}_{\GG_{n}}(f)$ , it suffices to take $N = |\TT_{n}|$
and $I=\GG_{n}$ in the Propositions \ref{prop:concentration_Xi} and
\ref{prop:concentration_deltai}. But for $q\geq 1$, the
concentration inequalities obtained thanks to these results are not
satisfactory. In the sequel, we will be interested in obtaining relevant
concentration inequalities for the empirical means
$\overline{M}_{\GG_{n}}(f)$ and $\overline{M}_{\TT_{n}}(f)$ when $q \geq 1$. 
\end{remark}


\subsection{Gaussian concentration inequalities for the empirical
means $\overline{M}_{\GG_{n}}(f)$ and $\overline{M}_{\TT_{n}}(f)$}

Throughout this section, we will focus only in the case $p=1$, and
will assume $(H_1(C))$. We set $r=r_0+r_1$.\\\medskip

The main goal of this subsection is to broaden the range of
application of deviation inequalities of $\overline{M}_{\GG_{n}}(f)$
and $\overline{M}_{\TT_{n}}(f)$ to cases where $r>1$, namely when it
is possible that one of the two marginal Markov chains is not a
strict contraction. The transportation inequality of Theorem
\ref{prop:trans_bmc} is a very powerful tool to get deviation
inequalities for all lipschitzian functions of the whole trajectory
(up to generation $n$), and may thus concern for example
Lipschitzian function of only offspring generated by $P_0$ or $P_1$.
Consequently, to get "consistent" deviation inequalities, both
marginal
Markov chains have to be contractions in Wasserstein distance.\\
However when dealing with $\overline{M}_{\GG_{n}}(f)$ or
$\overline{M}_{\TT_{n}}(f)$, we may hope for an averaging effect,
i.e. if one is not a contraction and the other one a strong
contraction it may in a sense compensate. Such averaging effect have
been observed at the level of the LLN and CLT in
\cite{Guyon,DeGeMa12} but only asymptotically. Our purpose here will
be then to show that such averaging effect will also affect
deviation inequalities.\\\medskip

We will use, directly inspired by Bobkov-G\"otze's Laplace transform
control, what we call Gaussian Concentration property: for
$\kappa>0$, we will say that a random variable $X$ satisfies
$GC(\kappa)$ if
\begin{equation*}
\EE\left[\exp\left(t\left(X-\EE\left[X\right]\right)\right)\right]
\leq \exp\left(\kappa t^{2}/2\right) \quad \forall t\in\RR.
\end{equation*}

Using Markov's inequality and optimization, this Gaussian
concentration property immediately implies that
$$\PP(X-\EE(X)\ge r)\le e^{-\frac{r^2}{2\kappa}}.$$
We may thus focus here only on the Gaussian concentration property $(GC)$.

\begin{prop}\label{Prop:GC_G}
Let $f$ be a real Lipschitzian function on $E$ and $n\in\NN$. Assume
that $(H_{1}(C))$ holds. Then $\overline{M}_{\GG_{n}}(f)$ satisfies
$GC(\gamma_{n})$ where
\begin{equation*} \displaystyle
\gamma_{n} = \begin{cases}
\frac{2C\|f\|_{Lip}^2}{|\GG_{n}|}\left(\frac{1
- \left(r^{2}/2\right)^{n+1}}{1-r^{2}/2}\right) \hspace{0.7cm} \text{if} \quad r\neq \sqrt{2} \\
\frac{2C\|f\|^2_{Lip}(n+1)}{|\GG_{n}|}  \hspace{2.4cm} \text{if}
\quad r = \sqrt{2}.
\end{cases}
\end{equation*}
We recall that here $r = r_{0} + r_{1}$.
\end{prop}

\begin{remark}
One can observe that for $r < \sqrt{2}$, the previous inequalities
are on the same order of magnitude that the inequalities obtained
thanks to Proposition \ref{prop:concentration_Xi} with $q<1$. For $r
< 2$ the above inequalities remain relevant since we just have a
negligible loss with respect to $|\GG_{n}|$. But for $r \geq \sqrt{2}$,
these inequalities are not significant (see the same type of limitations at the CLT level in \cite{DeGeMa12}).
\end{remark}

\begin{proof}
Let $f$ be a real Lipschitzian function on $E$, $n\in\NN$ and
$t\in\RR$. We have
\begin{multline*}
\EE\left[\exp\left(t2^{-n}\sum_{i\in\GG_{n}} f(X_{i})\right)\right]
= \EE\left[\exp\left(t2^{-n}\sum_{i\in\GG_{n-1}} (P_{0} +
P_{1})f(X_{i})\right)\right. \\ \times
\left.\EE\left[\exp\left(t2^{-n} \sum_{i\in\GG_{n-1}}
\left(f(X_{2i}) + f(X_{2i+1}) - (P_{0} +
P_{1})f(X_{i})\right)\right)\Big|\Ff_{n-1}\right]\right].
\end{multline*}
Thanks to the Markov property, we have
\begin{multline*}
\EE\left[\exp\left(t2^{-n} \sum_{i\in\GG_{n-1}} \left(f(X_{2i}) +
f(X_{2i+1}) - (P_{0} +
P_{1})f(X_{i})\right)\right)\Big|\Ff_{n-1}\right] \\ =
\prod_{i\in\GG_{n-1}} P\left(\exp\left(t2^{-n}\left(f\oplus f -
(P_{0} + P_{1})f\right)\right)\right)(X_{i})
\end{multline*}
where $f\oplus f$ is the function on $E^{2}$ defined by $f\oplus
f(x,y) = f(x) + f(y)$. We recall that from $(H_{1}(C))$ we have
$P(x,\cdot,\cdot)\in T_{1}(C)$ for all $x\in E$. Now, thanks to
Theorem \ref{thm:bobkov-gotze}, we have
\begin{multline*}
\prod_{i\in\GG_{n-1}} P\left(\exp\left(t2^{-n}\left(f\oplus f -
(P_{0} + P_{1})f\right)\right)\right)(X_{i}) \\ \leq
\prod_{i\in\GG_{n-1}} \exp\left(\frac{t^{2}C\|f\oplus
f\|_{Lip}^{2}}{2\times 2^{2n}}\right).
\end{multline*}
Since $\|f\oplus f\|_{Lip} \leq 2\|f\|_{Lip}$, we are led to
\begin{multline*}
\EE\left[\exp\left(t2^{-n}\sum_{i\in\GG_{n}} f(X_{i})\right)\right]
\leq \exp\left(\frac{2^{2}t^{2}2^{n-1}C\|f\|_{Lip}^{2}}{2\times
2^{2n}}\right)
\\ \times \EE\left[\exp\left(t2^{-n}\sum_{i\in\GG_{n-1}} (P_{0} +
P_{1})f(X_{i})\right)\right].
\end{multline*}
Doing the same for $\EE[\exp(t2^{-n}\sum_{i\in\GG_{n-1}}
(P_{0} + P_{1})f(X_{i}))]$ with $(P_{0} + P_{1})f$ replacing $f$ and
using the inequality
\begin{equation*}
\|(P_{0} + P_{1})f \oplus (P_{0} + P_{1})f\|_{Lip} \leq
2r\|f\|_{Lip},
\end{equation*}
we are led to
\begin{multline*}
\EE\left[\exp\left(t2^{-n}\sum_{i\in\GG_{n}} f(X_{i})\right)\right]
\leq \EE\left[\exp\left(t2^{-n}\sum_{i\in\GG_{n-2}} (P_{0} +
P_{1})^{2}f(X_{i})\right)\right] \\ \times
\exp\left(\frac{2^{2}t^{2}C\|f\|_{Lip}^{2}2^{n-1}}{2\times
2^{2n}}\right)\exp\left(\frac{2^{2}t^{2}C\|f\|_{Lip}^{2}r^{2}2^{n-2}}{2\times
2^{2n}}\right).
\end{multline*}
Iterating this method and using the inequalities
\begin{equation*}
\|(P_{0} + P_{1})^{k}f \oplus (P_{0} + P_{1})^{k}f\|_{Lip} \leq
2r^{k}\|f\|_{Lip} \quad \forall k\in\{1,\cdots,n-1\},
\end{equation*}
we obtain
\begin{multline*}
\EE\left[\exp\left(t2^{-n}\sum_{i\in\GG_{n}} f(X_{i})\right)\right]
\leq \exp\left(\frac{2^{2}t^{2}C\|f\|_{Lip}^{2}}{2\times 2^{2n}}
\sum_{k=0}^{n-1} r^{2k}2^{n-k-1}\right) \\ \times
\EE\left[\exp\left(t2^{-n}(P_{0} + P_{1})^{n}f(X_{1})\right)\right].
\end{multline*}
Since $\EE\left[t2^{-n}(P_{0} + P_{1})^{n}f(X_{1})\right] =
\EE\left[t2^{-n}\sum_{i\in\GG_{n}} f(X_{i})\right] = t2^{-n} \nu
(P_{0} + P_{1})^{n}f$, we obtain
\begin{multline*}
\EE\left[\exp\left(t2^{-n}\left(\sum_{i\in\GG_{n}} f(X_{i}) - \nu
(P_{0} + P_{1})^{n}f\right)\right)\right]\\ \leq
\exp\left(\frac{2^{2}t^{2}C\|f\|_{Lip}^{2}}{2\times 2^{2n}}
\sum_{k=0}^{n-1} r^{2k}2^{n-k-1}\right) \\ \times
\EE\left[\exp\left(t2^{-n}\left((P_{0} + P_{1})^{n}f(X_{1})\right) -
\nu (P_{0} + P_{1})^{n}f\right)\right].
\end{multline*}
Thanks to $(H_{1}(C))$, we conclude that
\begin{multline*}
\EE\left[\exp\left(t2^{-n}\left(\sum_{i\in\GG_{n}} f(X_{i}) - \nu
(P_{0} + P_{1})^{n}f\right)\right)\right]\\ \leq
\exp\left(\frac{2^{2}t^{2}C\|f\|_{Lip}^{2}}{2\times 2^{2n}}
\sum_{k=0}^{n} r^{2k}2^{n-k-1}\right)
\end{multline*}
and the results of the Proposition then follow from this last
inequality.
\end{proof}
For the ancestor-offspring triangle $(X_{i},X_{2i},X_{2i+1})$, we
have the following result which can be seen as a consequence of the
Proposition \ref{Prop:GC_G}.
\begin{corollary}\label{cor:GC_Gdelta}
Let $f$ be a real Lipschitzian function on $E^{3}$ and $n\in\NN$.
Assume that $(H_{1}(C))$ holds. Then $\overline{M}_{\GG_{n}}(f)$
satisfies $GC(\gamma'_{n})$ where
\begin{equation*} \displaystyle
\gamma'_{n} = \begin{cases} \frac{2C (1 + q)^{2}
\|f\|_{Lip}^2}{r^{2} |\GG_{n}|} \left(\frac{1 -
\left(r^{2}/2\right)^{n + 2}}{1 -
r^{2}/2}\right) \hspace{0.7cm} \text{if} \quad r\neq \sqrt{2} \\
\frac{2C(1+q)^{2}\|f\|^2_{Lip}(n + 2)}{|\GG_{n}|}  \hspace{2.4cm}
\text{if} \quad r = \sqrt{2}.
\end{cases}
\end{equation*}
\end{corollary}

\begin{proof}
Let $f$ be a real Lipschitzian function on $E^{3}$, $n\in\NN$ and
$t\in\RR$. We have
\begin{multline*}
\EE\left[\exp\left(t2^{-n}\sum_{i\in\GG_{n}}
f(X_{i},X_{2i},X_{2i+1})\right)\right] =
\EE\left[\exp\left(t2^{-n}\sum_{i\in\GG_{n}} Pf(X_{i})\right)\right.
\\ \times \left.\EE\left[\exp\left(t2^{-n}\sum_{i\in\GG_{n}}
\left(f(X_{i},X_{2i},X_{2i+1}) -
Pf(X_{i})\right)\right)\Big|\Ff_{n}\right]\right].
\end{multline*}
By the Markov property and thanks to the Proposition
\ref{prop:trans_bmc} and the Theorem \ref{thm:bobkov-gotze}, we have
\begin{multline*}
\EE\left[\exp\left(t2^{-n}\sum_{i\in\GG_{n}}
\left(f(X_{i},X_{2i},X_{2i+1}) -
Pf(X_{i})\right)\right)\Big|\Ff_{n}\right] \\ \leq
\exp\left(\frac{t^{2}C\|f\|_{Lip}^{2}2^{n}}{2\times 2^{2n}}\right).
\end{multline*}
Now, using $Pf$ instead of $f$ in the proof of the Proposition
\ref{Prop:GC_G} and using the fact that $\|Pf\|_{Lip}\leq
(1+q)\|f\|_{Lip}$ and
\begin{equation*}
\EE\left[2^{-n}\sum_{i\in\GG_{n}} f(X_{i},X_{2i},X_{2i+1})\right] =
\EE\left[2^{-n}\sum_{i\in\GG_{n}} Pf(X_{i})\right] =
2^{-n}\nu(P_{0}+P_{1})^{n}Pf,
\end{equation*}
we are led to
\begin{multline*}
\EE\left[\exp\left(t2^{-n}\left(\sum_{i\in\GG_{n}}
f(X_{i},X_{2i},X_{2i+1}) - \nu\left(P_{0} + P_{1}
\right)^{n}Pf\right)\right)\right] \\ \leq
\exp\left(\frac{4t^{2}C(1+q)^{2}\|f\|_{Lip}^{2}}{2^{2}\times 2^{n}}
\sum_{k=-1}^{n} \left(\frac{r^{2}}{2}\right)^{k}\right).
\end{multline*}
The results then follow by easy calculations.
\end{proof}
For the subtree $\TT_{n}$, we have the following result.

\begin{prop}\label{prop:GC_T}
Let $f$ be a real Lipschitzian function on $E$ and $n\in\NN$. Assume
that $(H_{1}(C))$ holds. Then $\overline{M}_{\TT_{n}}(f)$ satisfies
$GC(\tau_{n})$ where
\begin{equation*} \displaystyle
\tau_{n} = \begin{cases}
\frac{2C\|f\|_{Lip}^2}{(r-1)^{2}|\TT_{n}|}\left(1 + \frac{1 -
\left(r^{2}/2\right)^{n+1}}{1-r^{2}/2}\right) \hspace{1.7cm}
\text{if \, \, $r\neq \sqrt{2}$, $r\neq 1$} \vspace{0.1cm} \\
\frac{2C\|f\|^2_{Lip}}{(r-1)^{2}|\TT_{n}|} (r^{2}(n+1) +1 )
\hspace{2.4cm} \text{if \, \, $r = \sqrt{2}$} \vspace{0.2cm} \\
\frac{2C\|f\|^2_{Lip}}{|\TT_{n}|^{2}}\left(|\TT_{n}| -
\frac{n+1}{2}\right) \hspace{2.8cm} \text{if \, \, $r=1$}.
\end{cases}
\end{equation*}
\end{prop}

\begin{proof}
Let $f$ be a real Lipschitzian function on $E$ and $n\in\NN$. Note
that
\begin{equation*}
\EE\left[\sum_{i\in\TT_{n}} f(X_{i})\right] =
\nu\left(\sum_{m=0}^{n} (P_{0} + P_{1})^{m}f\right).
\end{equation*}
We have
\begin{multline*}
\EE\left[\exp\left(\frac{t}{|\TT_{n}|}\sum_{i\in\TT_{n}}
f(X_{i})\right)\right] =
\EE\left[\exp\left(\frac{t}{|\TT_{n}|}\sum_{i\in\TT_{n-2}}
f(X_{i})\right)\right. \\ \times
\exp\left(\frac{t}{|\TT_{n}|}\sum_{i\in\GG_{n-1}} \left(f +
\left(P_{0} + P_{1}\right)f\right)(X_{i})\right) \\ \times \left.
\EE\left[\exp\left(\frac{t}{|\TT_{n}|}\sum_{i\in\GG_{n-1}}
\left(f(X_{2i}) + f(X_{2i+1}) - (P_{0} +
P_{1})f(X_{i})\right)\right)\Big|\Ff_{n-1}\right]\right].
\end{multline*}
As in the proof of Proposition \ref{Prop:GC_G}, we have
\begin{multline*}
\EE\left[\exp\left(\frac{t}{|\TT_{n}|}\sum_{i\in\GG_{n-1}}
\left(f(X_{2i}) + f(X_{2i+1}) - (P_{0} +
P_{1})f(X_{i})\right)\right)\Big|\Ff_{n-1}\right] \\ \leq
\exp\left(\frac{2^{2}Ct^{2}\|f\|_{Lip}^{2}2^{n-1}}{2|\TT_{n}|^{2}}\right).
\end{multline*}
This leads us to
\begin{multline*}
\EE\left[\exp\left(\frac{t}{|\TT_{n}|}\sum_{i\in\TT_{n}}
f(X_{i})\right)\right] \leq
\exp\left(\frac{2^{2}Ct^{2}\|f\|_{Lip}^{2}2^{n-1}}{2|\TT_{n}|^{2}}\right)
\\ \times \EE\left[\exp\left(\frac{t}{|\TT_{n}|}\sum_{i\in\TT_{n-2}}
f(X_{i})\right) \exp\left(\frac{t}{|\TT_{n}|}\sum_{i\in\GG_{n-1}}
\left(f + \left(P_{0} + P_{1}\right)f\right)(X_{i})\right)\right].
\end{multline*}
Iterating this method, we are led to
\begin{multline*}
\EE\left[\exp\left(\frac{t}{|\TT_{n}|}\sum_{i\in\TT_{n}}
f(X_{i})\right)\right] \leq
\exp\left(\frac{2^{2}t^{2}C\|f\|_{Lip}^{2}}{2|\TT_{n}|^{2}}\sum_{k=0}^{n-1}
\left(\sum_{l=0}^{k} r^{l}\right)^{2} 2^{n-k-1}\right) \\ \times
\EE\left[\exp\left(\frac{t}{|\TT_{n}|}\sum_{m=0}^{n} \left(P_{0} +
P_{1}\right)^{m}f(X_{1})\right)\right]
\end{multline*}
and we then obtain thanks to \textit{(a)} of $(H_{1}(C))$ and
Theorem \ref{thm:bobkov-gotze}
\begin{multline*}
\EE\left[\exp\left(\frac{t}{|\TT_{n}|}\left(\sum_{i\in\TT_{n}}
f(X_{i}) - \nu\left(\sum_{m=0}^{n} (P_{0} +
P_{1})^{m}f\right)\right)\right)\right] \\ \leq
\exp\left(\frac{2^{2}t^{2}C\|f\|_{Lip}^{2}}{2|\TT_{n}|^{2}}\sum_{k=0}^{n}
\left(\sum_{l=0}^{k} r^{l}\right)^{2} 2^{n-k-1}\right).
\end{multline*}
In the last inequality we have used
\begin{equation*}
\left\|\sum_{m=0}^{n} (P_{0} + P_{1})^{m}f\right\|_{Lip} \leq
\left(\sum_{k=0}^{n} r^{k}\right) \|f\|_{Lip}.
\end{equation*}
The results then easily follows.
\end{proof}

For the ancestor-offspring triangle we have the following results
which can be seen as a consequence of the Proposition
\ref{prop:GC_T}.

\begin{corollary}\label{cor:GC_Tdelta}
Let $f$ be a real Lipschitzian function on $E^{3}$ and $n\in\NN$.
Assume that $(H_{1}(C))$ holds. Then $\overline{M}_{\TT_{n}}(f)$
satisfies $GC(\tau'_{n})$ where
\begin{equation*} \displaystyle
\tau'_{n} = \begin{cases}
\frac{2^{3}C(1+q)^{2}\|f\|_{Lip}^2}{|\TT_{n}|}\left(1 +
\frac{1}{(r-1)^{2}}\left(1 + \frac{r^{2}\left(1 -
\left(r^{2}/2\right)^{n+1}\right)}{1-r^{2}/2}\right)\right)
\hspace{0.05cm}
\text{if $r\neq \sqrt{2}$, $r\neq 1$} \vspace{0.1cm} \\
\frac{2^{3}C(1+q)^{2}\|f\|^2_{Lip}}{|\TT_{n}|} \left(1 + \frac{1 +
r^{2}(n + 1)}{(r-1)^{2}} \right)
\hspace{3.48cm} \text{if $r = \sqrt{2}$} \vspace{0.2cm} \\
\frac{2^{3}C(1+q)^{2}\|f\|^2_{Lip}}{|\TT_{n}|^{2}}\left(2|\TT_{n}| -
\frac{n+1}{2}\right) \hspace{3.8cm} \text{if $r=1$}.
\end{cases}
\end{equation*}
\end{corollary}

\begin{proof}
Let $f$ be a real Lipschitzian function on $E^{3}$ and $n\in\NN$. By
H\"{o}lder inequality and using the fact that
\begin{equation*}
\EE\left[\sum_{i\in\TT_{n}} f(\Delta_{i})\right] =
\EE\left[\sum_{i\in\TT_{n}} Pf(X_{i})\right],
\end{equation*}
we have
\begin{multline*}
\EE\left[\exp\left(\frac{t}{|\TT_{n}|}\left(\sum_{i\in\TT_{n}}
f(\Delta_{i}) - \EE\left[\sum_{i\in\TT_{n}}
f(\Delta_{i})\right]\right)\right)\right] \\ \leq
\left(\EE\left[\exp\left(\frac{2t}{|\TT_{n}|}\left(\sum_{i\in\TT_{n}}
\left(f(\Delta_{i}) -
Pf(X_{i})\right)\right)\right)\right]\right)^{1/2} \\ \times
\left(\EE\left[\exp\left(\frac{2t}{|\TT_{n}|}\left(\sum_{i\in\TT_{n}}
Pf(X_{i}) - \EE\left[\sum_{i\in\TT_{n}}
Pf(X_{i})\right]\right)\right)\right]\right)^{1/2}.
\end{multline*}
We bound the first term of the right hand side of the previous
inequality by using the same calculations as in the first iteration
of the proof of Corollary \ref{cor:GC_Gdelta}. We then have
\begin{multline*}
\left(\EE\left[\exp\left(\frac{2t}{|\TT_{n}|}\left(\sum_{i\in\TT_{n}}
\left(f(\Delta_{i}) -
Pf(X_{i})\right)\right)\right)\right]\right)^{1/2} \\ \leq
\exp\left(\frac{2t^{2}C\|f\|_{Lip}^{2}|\TT_{n}|}{2|\TT_{n}|^{2}}\right).
\end{multline*}
For the second term, we use the proof of the Proposition
\ref{prop:GC_T} with $Pf$ instead of $f$. We then have
\begin{multline*}
\left(\EE\left[\exp\left(\frac{2t}{|\TT_{n}|}\left(\sum_{i\in\TT_{n}}
Pf(X_{i}) - \EE\left[\sum_{i\in\TT_{n}}
Pf(X_{i})\right]\right)\right)\right]\right)^{1/2} \\ \leq
\exp\left(\frac{2^{3}t^{2}R(1+q)^{2}\|f\|_{Lip}^{2}}{2|\TT_{n}|^{2}}
\sum_{k=0}^{n} \left(\sum_{l=0}^{k}
r^{l}\right)^{2}2^{n-k-1}\right).
\end{multline*}
The results then follow by easy analysis and this ends the proof.
\end{proof}


\subsection{Deviation inequalities towards the invariant measure of the randomly drawn chain}

All the previous results do not assume any "stability" of the Markov
chain on the binary tree, whereas for usual asymptotic theorem the
convergence is towards mean of the function with respect to the
invariant probability measure of the random lineage chain. To
reinforce this asymptotic result by non asymptotic deviation
inequality, it is thus fundamental to be able to replace for example
$\EE(\overline{M}_{\TT_n}(f))$ by some asymptotic quantity. This
random lineage chain is a Markov chain with transition kernel $Q =
(P_0 + P_1)/2$. We shall now suppose the existence of a probability
measure $\pi$ such that $\pi Q=\pi$. We will consider a slight
modification of our main assumption and as we are mainly interested
in concentration inequalities, let us focus in the $p=1$ case:

\begin{assumption}[$H'_{1}(C)$]\label{ass:hp2} $\,$
\begin{itemize}
\item [(a)] $\nu\in T_{1}(C)$;
\item [(b)] $P_{b}(x,\cdot) \in T_{1}(C)$, $\forall x\in E$, $b=0,1$ ;
\item [(c)] $W_{1}^{d}(P(x,\cdot,\cdot),P(\tilde{x},\cdot,\cdot)) \leq
q \,d(x,\tilde{x})$, $\forall x, \tilde{x}\in E$ and some $q>0$. And
for $r_0,r_1>0$ such that $r_0+r_1<2$, for $b=0,1$,
$W_{1}^{d}(P_b(x,\cdot),P_b(\tilde{x},\cdot)) \leq r_b
\,d(x,\tilde{x})$, $\forall x, \tilde{x}\in E$.
\end{itemize}
\end{assumption}
Under this assumption, using the convexity of $W_1$ (see
\cite{Villani}), we easily see that
$$W_1(Q(x,\cdot),Q(\tilde x,\cdot))\le \frac{r_0+r_1}2 d\left(x,\tilde x\right),\qquad \forall x,\tilde x$$
ensuring the strict contraction of $Q$, and then the exponential
convergence towards $\pi$ in Wasserstein distance, namely (assuming
that $\pi$ has a first moment)
$$W_1(Q^n(x,\cdot),\pi)\le \left(\frac{r_0+r_1}2\right)^{n} \int d(x,y)\pi(dy).$$
Let us show that we may now control easily the distance between
$\EE(\overline{M}_{\TT_n}(f))$ and $\pi(f)$. Indeed, we may first
remark that
$$\EE\left(\sum_{k\in\GG_n}f(X_k)\right) = \nu(P_0 + P_1)^nf$$
so that assuming that $f$ is $1$-lipschitzian, and by the dual
version of the Wasserstein distance
\begin{eqnarray*}
\left| \EE(\overline{M}_{\TT_n}(f))-\pi(f)\right| &=&
\frac{1}{|\TT_n|}\left|\sum_{j=1}^n\EE\left(\sum_{k\in\GG_j}(f(X_k)-\pi(f)\right)\right|\\
&=&\frac{1}{|\TT_n|}\left|\sum_{j=1}^n2^j\nu\left(\frac{P_0+P_1}2\right)^j (f-\pi(f))\right|\\
&\le&\frac{1}{|\TT_n|}\sum_{j=1}^n2^j W_1(\nu Q^j,\pi)\\
&\le&\frac{1}{|\TT_n|}\sum_{j=1}^n(r_0+r_1)^j\\
&\le&c_n:=\left\{ \begin{array}{ll}
c\left(\frac{r_0+r_1}2\right)^{n+1}&\mbox{ if }r_0+r_1\not=1\\
c\frac{n}{2^{n+1}}&\mbox{ if }r_0+r_1=1\end{array}\right.
\end{eqnarray*}
for some universal $c$, which goes to 0 exponentially fast as soon as $r_0+r_1<2$ which was
assumed in $(H_1'(C))$. We may then see that for $r>c_n$
$$\PP\left(\overline{M}_{\TT_n}(f)-\pi(f)>r\right)\le \PP\left(\overline{M}_{\TT_n}(f)-\EE(\overline{M}_{\TT_n}(f))>r-c_n\right)$$
and one then applies the result of the previous subsection.


\section{Application to nonlinear bifurcating autoregressive models}

The setting will be here the case of the nonlinear
bifurcating autoregressive models. It has been
considered as a particular realistic model to study
cell aging \cite{SteMadPauTad}, and the asymptotic
behavior of parametric estimators as well as non parametric
estimators has been considered in an important series of work,
see e.g. \cite{BH99,BH00,BZ4,BZ105,BZ205, Guyon,  BDSGP,DM,SGM12,DSGPM11,BiDj12}
(and for example in the random coefficient setting in \cite{DeGeMa12}).\\

We will then consider the following model where to simplify the
state space $E=\RR$, where $\Ll(X_1) = \mu_0$ satisfies $T_1$ and we
recursively define on the binary tree as before
\begin{equation}
\left\{\begin{array}{l}
X_{2k}\hspace{0.4cm}=f_0(X_k)+\varepsilon_{2k}\\
X_{2k+1}=f_1(X_k)+\varepsilon_{2k+1}
\end{array}\right.
\end{equation}
with the following assumptions:

\begin{assumption}[NL]\label{ass:NL}
$f_0$ and $f_1$ are Lipschitz continuous function.
\end{assumption}

\begin{assumption}[No]\label{ass:No}
$(\varepsilon_{k})_{k\ge1}$ are centered i.i.d.r.v. and for all
$k\ge0$, $\varepsilon_k$ have law $\mu_\varepsilon$ and satisfy for
some positive $\delta_\varepsilon$,
$\mu_{\varepsilon}\left(e^{\delta_\varepsilon x^2}\right)<\infty$.
Equivalently, $\mu_\varepsilon$ satisfies $T_1(C_\varepsilon)$.
\end{assumption}

It is then easy to deduce that under these two assumptions, we
perfectly match with the previous framework. Denoting $P_0$ and
$P_1$ as previously, we see that $(H'_1)$ is verified, with the
additional fact that $P=P_0\otimes P_1$. We will do the proof for
$P_0$, being the same for $P_1$. The conclusion follows for $P$ by
conditional independence of $X_{2k}$ and $X_{2k+1}$. Let us first
prove that $P_0(x,\cdot)$ satisfies $T_1$. Indeed $P_0(x,\cdot)$ is
the law of $f_0(x)+\varepsilon_{2k}$, and we have thus to verify the
Gaussian integrability property of Theorem \ref{thm:dgw}. To this
end, consider $x_0=f(x)$, and choose $\delta_{\varepsilon}$ of
condition (No) to verify
the Gaussian integrability property. We have thus that $P_{0}$ satisfies $T_1(C_P)$.\\
We prove now the Wasserstein contraction property. $P_0(x,\cdot)$ is
of course the law of $f_0(x)+\varepsilon_k$. Here $\varepsilon_k$
denotes a generic random variable and thus the law of $P_0(y,\cdot)$
is the law of $f_0(y)+\varepsilon_k$ and an upper bound of the
Wasserstein distance between $P_0(x,\cdot)$ and $P_0(y,\cdot)$ can
then be obtained by the coupling where we really choose the same
noise $\varepsilon_k$ for the realization of the two marginal laws
so that

Let $f$ be any Lipschitz function such that $\|f\|_{Lip}\leq 1$ 
\begin{eqnarray*}
\left|\int_S f(z) P_0(x,dz)-\int_S f(z) P_0(y,dz)\right| &=&\mathbb{E}\left[f\left(f_0(x)+\varepsilon_1\right)-f\left(f_0(y)+\varepsilon_1\right)\right]\\
&\leq& \|f\|_{Lip}|f_0(x)-f_0(y)|.
\end{eqnarray*}
By the Monge-Kantorovitch duality expression of the Wasserstein distance, one has then
$$W_1(P_0(x,\cdot),P_0(y,\cdot))\le |f_0(x)-f_0(y)|
\le \|f_0\|_{Lip}|x-y|.$$

Thus under (NL) and (No), our model fits
in the framework  of the previous section with
$q=\|f_0\|_{Lip}+\|f_1\|_{Lip}$, $r_0=\|f_0\|_{Lip}$ and
$r_1=\|f_1\|_{Lip}$. We will be interested here in the non
parametric estimation of the autoregression functions $f_0$ and
$f_1$, and we will use Nadaraya-Watson kernel type estimator, as
considered in \cite{BO}. Let $K$ be a kernel satisfying the
following assumption.
\begin{assumption}[Ker]\label{ass:ker}
The function $K$ is non negative, has compact support $[-R,R]$, is
Lipschitz continuous with constant $\| K\|_{Lip}$ and such that
$\int K(z)dz=1$.
\end{assumption}

Let us also introduce as usual a bandwidth $h_n$ which will be taken
to simplify as $h_n:=|\TT_n|^{-\alpha}$ for some $0<\alpha<1$. The
Nadaraya-Watson estimators are then defined as for $x\in\RR$
$$\widehat f_{0,n}(x) :=
\frac{\displaystyle \frac{1}{|\TT_n|h_n}
\sum_{k\in\TT_n}K\left(\frac{X_k-x}{h_n}\right)\,X_{2k}}
{\displaystyle \frac{1}{|\TT_n|h_n}
\sum_{k\in\TT_n}K\left(\frac{X_k-x}{h_n}\right)}$$
$$\widehat f_{1,n}(x):=\frac{\displaystyle \frac{1}{|\TT_n|h_n}
\sum_{k\in\TT_n}K\left(\frac{X_k-x}{h_n}\right)\,X_{2k+1}}
{\displaystyle \frac{1}{|\TT_n|h_n}\sum_{k\in\TT_n}
K\left(\frac{X_k-x}{h_n}\right)}.$$ Let us focus on $f_0$, as it
will be exactly the same for $f_1$ and fix $x\in\RR$. We will be
interested here in deviation inequalities of $\widehat f_{0,n}(x)$
with respect to $f(x)$. One has to face two problems. First it is an
autonormalized estimator. It will be dealt with considering
deviation inequalities for the numerator and denominator separately
and reunite them. Secondly $(x,y)\to K(x)y$ is in fact not
Lipschitzian in general state space, so that the result of the
previous section for deviation inequalities of Lipschitzian function
of ancestor-offspring may not be applied directly. Let us tackle
this problem. By definition
\begin{eqnarray*}
\widehat f_{0,n}(x)-f(x)&=&\frac {\displaystyle
\frac{1}{|\TT_n|h_n}\sum_{k\in\TT_n}
K\left(\frac{X_k-x}{h_n}\right)\,
[f_0(X_{k})-f_0(x)+\varepsilon_{2k}]}
{\displaystyle \frac{1}{|\TT_n|h_n}\sum_{k\in\TT_n}K\left(\frac{X_k-x}{h_n}\right)}\\
&:=&\frac{N_n+M_n}{D_n}.
\end{eqnarray*}
where
$$N_n:=\sum_{k\in\TT_n}K\left(\frac{X_k-x}{h_n}\right)\,[f_0(X_{k})-f_0(x)],$$
$$M_n:=\sum_{k\in\TT_n}K\left(\frac{X_k-x}{h_n}\right)\,\varepsilon_{2k},$$
$$ D_n=\sum_{k\in\TT_n}K\left(\frac{X_k-x}{h_n}\right).$$
Denote also $\tilde N_n = N_n/(|\TT_n|h_n)$, $\tilde
M_n=M_n/(|\TT_n|h_n)$, $\tilde D_n=D_n/(|\TT_n|h_n)$. Let us
remark that $D_n$ and $M_n$ completely enter the framework of
Proposition \ref{prop:GC_T}. We may thus prove
\begin{prop}\label{prop:dev_ker}
Let us assume that $(NL)$, $(No)$ and $(Ker)$ holds, and $
q=\|f_0\|_{Lip}+\|f_1\|_{Lip}<\sqrt{2}$. Let us also suppose that
$\alpha<1/4$. Then for all $r>0$ such that $r>\EE(\tilde
N_n)/\EE(\tilde D_n)$, there exists constants $C,C',C''>0$ 
such that
\begin{eqnarray*}\PP\left(|\widehat f_{0,n}(x)-f(x)|>r\right)& \le &2\exp\left(-C(r\EE(\tilde D_n)-\EE(\tilde N_n))^2
|\TT_{n}|h_{n}^{2}\right)\\&&+2\exp\left(-C'\frac{(r\EE(\tilde D_n)-\EE(\tilde N_n))^2|\TT_{n}|h_{n}^{2}}{1+C''\frac{r^2}{h_n^2}}\right)
.
\end{eqnarray*}
\end{prop}

\begin{proof}
Remark first that, by $(Ker)$, $K$ is Lipschitz continuous so that
$y\to K(\frac{y-x}{h_n})$ is also lipschitzian with constant $\|
K\|_{Lip}/h_n$. The mapping $y\to
K(\frac{y-x}{h_n})(f_0(y)-f_0(x))$, as $K$ has a compact support and
$f_0$ is Lipschitzian, is also Lipschitzian with constant
$R\|K\|_{Lip}\|f_0\|_{Lip}+\|f_{0}\|_{Lip}\|K\|_{\infty}$. We can
then use Proposition \ref{prop:GC_T} to get deviation inequalities
for $D_n$. For all positive $r$ there exists a constant $L$
(explicitly given through Proposition \ref{prop:GC_T}), such that
$$\PP(|D_n-\EE(D_n)|> r |\TT_n|h_n)\le 2\exp\left(-Lr^2|\TT_n|h_n^4/ \| K\|_{Lip}^2\right).$$
For $N_n+M_n$ we cannot directly apply Proposition \ref{prop:GC_T}
due to the successive dependence of $X_k$ at generation $n$ and
$\varepsilon_{2k}$ of generation $n-1$. But as we are interested in
deviation inequalities, we may split the deviation coming from each
term.  For $N_n$, it is once again a simple application of
Proposition \ref{prop:GC_T},
$$\PP(|N_n-\EE(N_n)|> r|\TT_n|h_n)\le
2\exp\left(
\frac{-Lr^2|\TT_n|h_n^2}{(R\|K\|_{Lip}\|f_0\|_{Lip}+\|f_{0}\|_{Lip}\|K\|_{\infty})^2}\right).$$

Note that $\varepsilon_{2k}$ is independent of $X_k$, and centered
so that $\EE(M_n)=0$, and satisfies a transportation inequality.
Note also  that $K$ is bounded. By simple conditioning argument, we
may control the Laplace transform of $M_n$ quite simply. We then
have for all positive $r$
$$\PP(|M_n|>r|\TT_n|h_n)\le 2
\exp\left(-r^2\frac{|\TT_n|h_n^2\|K\|^2_\infty}{2C\varepsilon}\right).$$

However, we cannot use directly these estimations as the estimator is autonormalized. Instead
\begin{eqnarray*}
&&\PP\left(\widehat f_{0,n}(x)-f(x)>r\right)\\
&&\qquad\le\PP(\tilde{
N}_{n}+\tilde{M}_{n}>r \tilde{D}_n)\\
&&\qquad\le \PP\left(\tilde{N}_{n}-\EE(\tilde{N}_{n})-r(\tilde{D_n}-\EE(\tilde{D}_n))+\tilde{M}_{n}>r \EE(\tilde{D_n})-\EE(\tilde{N}_{n})\right)\\
&&\qquad\le \PP\left(\tilde{N}_{n}-\EE(\tilde{N}_{n})-r(\tilde{D_n}-\EE(\tilde{D}_n)>(r \EE(\tilde{D_n})-\EE(\tilde{N}_{n}))/2\right)\\
&&\qquad \hspace*{4.45cm} + \PP\left(\tilde{M}_n>(r
\EE(\tilde{D_n})-\EE(\tilde{N}_{n}))/2\right)
\end{eqnarray*}
Remark now to conclude that $K((y-x)/h_n)(f(y)-f(x))+K((y-x)/h_n)$
is $(R\|K\|_{Lip}\|f_0\|_{Lip} + \|f_{0}\|_{Lip}\|K\|_{\infty} + r\|
K\|_{Lip}/h_n)$-Lipschitzian, and we may then proceed as before.
\end{proof}
\begin{remark}
In order to get fully practical deviation inequalities, let us
remark that
\begin{equation*}
\EE\left[\tilde D_n\right] = \frac{1}{|\TT_{n}|h_{n}}
\sum_{m=0}^{n} 2^{m}\mu_{0} Q^{m}H {\underset{n\rightarrow +
\infty}{{\longrightarrow}} } \nu(x)
\end{equation*}
where $H(y) = K((y-x)/h_{n})$, $\nu(\cdot)$ is the invariant density
of the Markov chain associated to a random lineage and
\begin{equation*}
\EE\left[\tilde N_n\right] = \frac{1}{|\TT_{n}|h_{n}}
\sum_{m=0}^{n} 2^{m}\left(\mu_{0}Q^{m}(Hf_{0}) -
f_{0}(x)\mu_{0}Q^{m}H\right) {\underset{n\rightarrow +
\infty}{{\longrightarrow}} } 0.
\end{equation*}
We refer to \cite{BO} for quantitative versions of these limits.\end{remark}

\begin{remark}
Of course this non parametric estimation is in some sense
incomplete, as we would have liked to consider a deviation
inequality for $\sup_{x}|\hat f_{0,n}(x)-f_0(x)|$. The problem is
somewhat much more complicated here, as the estimator is self
normalized. However, it is a crucial problem that we will consider
in the near future. For some ideas which could be useful here, let
us cite the results of \cite{BGV} for (uniform) deviation
inequalities for estimators of density in the i.i.d. case, and to
\cite{FG} for control of the Wasserstein distance of the empirical
measure of i.i.d.r.v. or of Markov chains.
\end{remark}


\begin{remark}[Estimation of the $\TT$-transition probability]
We assume that the process has as initial law, the invariant probability
$\nu$. We denote by $f$ the density of $(X_{1},X_{2},X_{3})$. For
the estimation of $f$, we propose the estimator $\widehat{f}_{n}$
defined by
\begin{equation*}
\widehat{f}_{n}(x,y,z) = \frac{1}{|\TT_{n}|h_{n}} \sum_{k\in\TT_{n}}
K\left(\frac{x - X_{k}}{h_{n}}\right) K\left(\frac{y -
X_{2k}}{h_{n}}\right) K\left(\frac{z - X_{2k+1}}{h_{n}}\right).
\end{equation*}
An estimator of the $\TT$-probability transition is then given by
\begin{equation*}
\widehat{P}_{n}\left(x,y,z\right) =
\frac{\widehat{f}_{n}(x,y,z)}{\tilde D_n}.
\end{equation*}
For $x,y,z\in\RR$, one can observe that the function $G$ defined on
$\RR^{3}$ by
\begin{equation*}
G(u,v,w) = K\left(\frac{x - u}{h_{n}}\right) K\left(\frac{y -
v}{h_{n}}\right) K\left(\frac{z - w}{h_{n}}\right)
\end{equation*}
is Lipschitzian with $\|G\|_{Lip} \leq
(\|K\|_{\infty}^{2}\|K\|_{Lip})/h_{n}$. We have
\begin{equation*}
\widehat{P}_{n}\left(x,y,z\right) - P(x,y,z) =
\frac{\widehat{f}_{n}(x,y,z) - f(x,y,z)}{\tilde{D}_{n}} +
\frac{f(x,y,z)(\nu(x) -
\tilde{D}_{n})}{\nu(x)\tilde{D}_{n}}.
\end{equation*}
Now using the decomposition
\begin{multline*}
\widehat{f}_{n}(x,y,z) - f(x,y,z) = \left(\widehat{f}_{n}(x,y,z) -
\EE\left[\widehat{f}_{n}(x,y,z)\right]\right) \\ +
\left(\EE\left[\widehat{f}_{n}(x,y,z)\right] - f(x,y,z)\right)
\end{multline*}
and the convergence of $\EE\left[\widehat{f}_{n}(x,y,z)\right]$ to
$f(x,y,z)$, we obtain a deviation inequality for
$|\widehat{P}_{n}\left(x,y,z\right) - P(x,y,z)|$ similar to that
obtained at the Proposition \ref{prop:dev_ker}.

When the density $g_{\varepsilon}$ of
$(\varepsilon_{2},\varepsilon_{3})$ is known, another strategy for
the estimation of the $\TT$-transition probability is to observe
that $P(x,y,z) = g_{\varepsilon}(y-f_{0}(x),z-f_{1}(x))$. An
estimator of $P(x,y,z)$ is then given by $\widehat{P}_{n}(x,y,z) =
g_{\varepsilon}(y-\widehat{f}_{0,n}(x),z-\widehat{f}_{1,n}(x))$
where $\widehat{f}_{0,n}$ and $\widehat{f}_{1,n}$ are estimators
defined above. If $g_{\varepsilon}$ is Lipschitzian, we have
\begin{equation*}
|\widehat{P}_{n}\left(x,y,z\right) - P(x,y,z)| \leq
\|g_{\varepsilon}\|_{Lip}\left(|\widehat{f}_{0,n}(x) - f_{0}(x)| +
|\widehat{f}_{1,n}(x) - f_{1}(x)|\right)
\end{equation*}
and the deviation inequalities for
$|\widehat{P}_{n}\left(x,y,z\right) - P(x,y,z)|$ are thus of the
same order that those given by the Proposition \ref{prop:dev_ker}.
\end{remark}

\bigskip

\noindent{\bf Acknowledgements}\\ S.V. Bitseki Penda sincerely thanks Labex Ma\-th\'ematiques Hadamard. A. Guillin was supported  by the French ANR STAB project.

\addcontentsline{toc}{section}{References}
\bibliography{transportbif}
\bibliographystyle{plain}
\end{document}